\documentclass[12pt,a4paper]{article}

\usepackage[margin=2cm]{geometry}

\usepackage{comment}
\usepackage[all]{xy}

\usepackage{authblk}

\usepackage{amsmath}
\usepackage{amssymb,amsthm}
\usepackage{wasysym}
\usepackage{amsfonts}
\usepackage{xcolor}

\usepackage{graphicx}

\usepackage{hyperref}
\usepackage{cleveref}
\usepackage{datetime}

\usepackage{layout}

\theoremstyle{definition}
\newtheorem{thm}{Theorem}[section]
\newtheorem*{thmm}{Main Theorem}
\newtheorem{proposition}[thm]{Proposition}
\newtheorem{lem}[thm]{Lemma}
\newtheorem{example}[thm]{Example}
\newtheorem{remark}[thm]{Remark}

\newtheorem{cor}[thm]{Corollary}
\newtheorem{notation}[thm]{Notation}

\newtheorem{obs}[thm]{Observation}
\newtheorem{defn}[thm]{Definition}
\newtheorem{fact}[thm]{Fact}
\newcommand{\Equiv}[1]{\stackrel{#1}{\equiv}}
\renewcommand{\bf}[1]{\mathbf{#1}}
\newcommand{\group}{\Gamma_2\oplus \Gamma_1}
\newcommand{\Group}{\Lambda_2\oplus \Lambda_1}

\newcommand{\rg}{L_{\text{ring}}}

\renewcommand{\circ}{%
	\mathrel{\raisebox{0.55mm}{\scalebox{0.75}{$\bigcirc$}}}}

\DeclareMathOperator{\tp}{p-tp}

\DeclareMathOperator{\diag}{diag}
\begin{document}
\title{Construction of a valued field whose valuation ring is definable but neither $\exists\forall\exists$ nor $ \forall\exists\forall$-definable in the language of rings}
\date{\today}
\author[1]{Mohsen Khani \thanks{\href{mailto:mohsen.khani@iut.ac.ir}{mohsen.khani@iut.ac.ir}}}
\author[2]{Shaghayegh Shirani}
\author[3]{Zahra Yadegari}
\author[4]{Afshin Zarei\thanks{This author was supported by a grant from IPM.\\ \hspace*{0.45cm}
2020\textit{ Mathematics Subject Classification.} Primary 
03C60, 03C40 Secondary 12L12, 12J10}}
\affil[1,2,3]{Department of Mathematical Sciences\\ Isfahan University of Technology\\ P.O. Box: 84156-83111, Isfahan, Iran}
\affil[4]{School of Mathematics\\ Institute for Research in Fundamental Sciences 
	(IPM)\\ P.O. Box: 19395-5746, Tehran\\ Iran.}
\maketitle
\begin{abstract}
	We give an example of a valued field $(K,A)$ such that the valuation ring $A$ is definable by an $L_{\text{ring}}$-formula without parameters, but there is no $\exists\forall\exists$ or $\forall\exists\forall$-formula  in $L_{\text{ring}}$ to define it.
\end{abstract}
\section{Introduction}
The question of  definability  of the valuation ring in a valued field—the language of rings, without parameters,  and with specified complexity—has recently attracted considerable attention from model theorists.
 A comprehensive survey on this subject can be found in \cite{Janke-survey}; below, we highlight some key results relevant to our study.
 Throughout, we denote a valued field $K$ with valuation ring $A$, value group $\Gamma$ and residue field $k$
 as a tuple $\mathcal{K}=(K,A,k,\Gamma)$, or sometimes shorter as $(K,A)$.
\par 
For the concrete case of the field of Hahn series $\mathbb{F}_q((t))$ 
it is shown in \cite{anscombe} that the valuation ring   $\mathbb{F}_q[[t]]$
is definable with relatively low complexity, that is by an existential parameter-free formula.
This result, along with the techniques used in its proof, is further generalized in \cite{fehm}, where it is shown that for any Henselian valued field 
$\mathcal{K}$,
if the residue field 
$k$
is finite or pseudo-algebraically closed and its algebraic part is not algebraically closed, then the ring 
$A$
is definable by an existential parameter-free formula. These two are instances of the general result  in
\cite{prestel}, that whenever
$\mathcal{K}$ is a Henselian valued field, then if  $k$  is finite or 
pseudo-finite, the valuation ring is always definable by an
$\exists\forall $-formula.
In \cite{prestel} 
 Prestel asks whether a parameter-free definable valuation ring in a Henselian valued field
 is always definable 
with an $\exists \forall$ or a $\forall \exists$-formula.
\par 
This question is answered negatively in \cite{main}
where a valued field is constructed, whose valuation ring is definable, but definable neither with a
$\forall\exists$ nor with an $\exists\forall$-formula. This paper, is a curious attempt to further examine ideas in 
\cite{main}
in order to construct  
 a valued field whose valuation ring is definable but
there is no  $\exists\forall\exists$ or $\forall \exists\forall$-formula  to define it.
So the main theorem of this article is the following:
\begin{thmm}
	There is a Henselian valued field whose valuation ring
	is definable, without parameters,  but  is not definable
	by any $\exists\forall\exists$ or $\forall \exists\forall$-formula.
\end{thmm}
The strategy we employ to prove the above theorem involves a careful modification of each step in \cite{main}.
In
Section \ref{section-definability},
 Lemma  \ref{general form of prestel theorem} and Corollary \ref{cor-pi-n+1-definition},
we give model theoretic criteria 
for the existence of 
a uniform $\exists\forall\exists$ or $\forall\exists\forall$-definition in $\rg$
for the unary predicate $p$ 
in all models $(M,p(M))$ 
of  the expansion of a field $M$ with a unary predicate $ p $. This will be a
modification of the  ``Characterization Theorem''  in
\cite{prestel}.
In Section \ref{ax-kochen}, Theorem \ref{ax- asli}, we give an
$\exists\forall$-version for the Ax-Kochen theorem, which is based on an adoption of 
{Ax-Kochen Theorem in
	 \cite{kuh}.}
\par 
 In Sections \ref{group-construction} we construct two groups, 
which we use in order to construct 
 a valued field $\mathcal{K}$
 as a field of Hahn Series.
  Finally, in 
 section \ref{final-step}
we use the mentioned theorems in order to prove that the valuation ring of the constructed valued field $\mathcal{K}$ is neither $\exists\forall\exists$ nor
$\forall\exists\forall$-definable.
	\section{Definability of a unary predicate by certain complexity}
	\label{section-definability}
	Let $T$ be a theory in a language $L\cup \{p\}$ where $p$ is a unary predicate symbol.
In this section, we aim to establish elementary model-theoretic criteria for the existence of $\Sigma_{n}$ and $\Pi_n$ formulas in  $L$ 
that define $p(M)$ uniformly  in all models $(M,p(M))$ of $T$. The results in this section are 
 generalizations of the results in \cite{prestel}.
 The main lemma of this section is Lemma \ref{general form of prestel theorem}. It is worth noting that
 for our main result, we actually need the easier direction of this lemma and its corollary.
  However we have proved both directions for 
 completeness. 
	\begin{defn}\hfill 
		\label{defn-sigma-n-closed}
		\begin{itemize}
			\item 
		Let $\Phi(x)$ be a set of formulas in the single variable $x$  and in the language $L$,  $M$ be an $L$-structure and
		$a\in M$. By $\tp^M_{\Phi}(a)$ we mean the set $\{\varphi(x)\in \Phi: M\models \varphi(a)\}$.
		\item 
		Assume that ${M}\subseteq {N}$ are two structures and $\Delta$ is a set of formulas.  We write  ${M}{\subseteq_{_{\Delta}}} {N}$, and say that
		$M$ is $\Delta$-closed in $N$, 
		whenever for all $\varphi(\bar{x})\in \Delta$ and all $\bar{a}\in M$, if
		${N}\models \varphi(\bar{a})$ then ${M}\models \varphi(\bar{a})$.
	\end{itemize}
\end{defn}
		 A special case of the second item is 
	that a structure $M$ is called $\Sigma_n$-closed in $N$ if for all
	$\Sigma_n$-formulas $\varphi(\bar{x})$ and all $\bar{a}\in M$ 
	if $N\models \varphi(\bar{a})$ then $M\models \varphi(\bar{a})$.
We will also use notations such as $M \subseteq_{\exists} N$ in the following sections, with their meaning made clear by the preceding definition.

	The following is
	a standard model-theoretic application of the compactness theorem, however for completeness we have included proof.
	\begin{lem}
		\label{lem1 new}
		Let $T$ be a first order theory in the language $L\cup \{p\}$, where $p$ is a unary predicate.
		Assume that $\Phi$ is a set of formulas 
		 in the language $L$,
		closed under
		$\wedge,\vee$ and
		with a single variable $x$. Assume that for each model
		$(M,p(M))$ of $T$ we have the following:
		\begin{align}
			\label{condition-for-definability-of-p}
			\text{For each $a,b\in M$  if $\tp^{M}_{\Phi}(a)\subseteq \tp^{M}_{\Phi}(b)$,
			then if $a\in p(M)$ then $b\in p(M)$.}
		\end{align}
		Then $p(M)$ is uniformly definable in all models
		$(M,p(M))$ of $T$
		 by a formula in $\Phi$.
	\end{lem}
	\begin{proof}
		 Condition
		\eqref{condition-for-definability-of-p}
		 above is equivalent to the following implication:
		\[
		T\cup  \{\big(\varphi(x)\to \varphi(y)\big):\varphi\in \Phi\}\models \big(p(x)\to p(y)\big).
		\]
		By compactness,  there
		are $\varphi_1,\ldots\varphi_n\in \Phi$ such that:
		\[
		T\cup \bigwedge_{i=1}^n (\varphi_i(x)\to \varphi_i(y))\models (p(x)\to p(y)).
		\]
		Put $\Phi'=\{\varphi_1,\ldots,\varphi_n\}$.
		Now let $M_1$ be a model of $T$. For each $a\in p(M_1)$ let $\Phi'(a)$ be the set
		$\{\varphi\in \Phi': M_1\models \varphi(a)\}$.
		 Since 
		 $\Phi'$
		 is finite, it follows that  the family
		  $\{\Phi'(a)\}_{a\in p(M_1)}$
		is also finite, and equal to, say $\Phi'(a_1)\cup \ldots\cup \Phi'(a_m)$.
		Thus we have:
		\[
		M_1\models p(x)\leftrightarrow \bigvee_{i=1}^m \bigwedge_{\varphi\in \Phi'(a_i)} \varphi(x).
		\]
Denote the formula 
on the right-hand side of the implication above by
 $\psi_{M_1}$. 
 This formula 
 defines $p(M_1)$ in $M_1$, 
but it may not define $p(M)$ in other models $(M,p(M))$ of $T$.
 To obtain a uniform formula, we have to apply the same steps to every model of $T$.
 So for each model $M\models T$ there is $\psi_M$ that defines
 $p(M)$ in $M$. Also the family 
 $\{\psi_M\}_{M\models T}$ is finite since each 
 $\psi_M$ is obtained using conjunctions and disjunctions of formulas in $\Phi'$.
  Thus, the following formula, defines $p$ uniformly in all models:
$$\chi(x) = \bigvee_{M\models T}{\psi_M}.$$
    	\end{proof}
    	\begin{lem}
    		\label{general form of prestel theorem}
    	Assume that $T$ is a theory in the language  $L\cup \{p\}$, where $p$ is a unary predicate. 
    	Then
    	$p(M)$ 
    	is definable by a
    	$\Sigma_{n+1}$-formula  uniformly in all models of $T$ if and only if the following holds:
    	\begin{quote}
    		For each pair of  models $(M,p(M))$ and $(M^*,p(M^*))$ of $T$ if $M$ is 	$\Sigma_n$-closed in $M^*$ in the language $L$, then
for each $a\in M$ if $M\models p(a)$ then $M^*\models p(a)$.
    	\end{quote}
    	    \end{lem}
    \begin{proof}	
    	The ``only if" direction is easy to verify and we only prove the ``if" direction. 
    Let $\Phi$ be the set of all one-variable  $\Sigma_{n+1} $-formulas in the language $ L $,
    which by changing the bounded variables if  necessary, we assume is closed under $\wedge$ and $\vee$.
     Our objective is to show that  the  assumption of Lemma \ref{lem1 new} holds. Assume, for a contradiction, that there exists a model $(M,a,b)$ for the the theory 
    $$ T \cup``{\tp_{\Phi}(c) \subseteq \tp_{\Phi}(d)}" \cup \{p(c) \wedge \neg p(d)\} $$
    in the extended language that contains constant symbols $c,d$.
    \par 
We first  claim that in this case, there also exists a model for the following theory:
    $$T^*= T \cup \text{diag}_{\Pi_{n}} (M)\cup \neg p(a),$$
    where by 
    $\text{diag}_{\Pi_{n}} (M)$ we mean the set of $\Pi_{n}$ sentences in the language $L(M)$ satisfied by $M$.
    To prove the claim, we need to demonstrate that each finite subset of $T^*$ is satisfiable.
    Therefore, let
    $$T' = T \cup \{\forall \bar{x}\  \exists \bar{y} \cdots\ \psi(\bar{x},\bar{y},a,a_1,\cdots,a_n)\} \cup\{\neg p(a)\}$$
    be a portion of the theory $T^*$ where the formula
    $ \forall \bar{x} \ \exists \bar{y}\cdots \ \psi(x,y,a,a_1,\cdots,a_n)$ is a (conjunction of) formula(s) in $\text{diag}_{\Pi_{n}}(M)$.  
 It is then clear that:
     $$(M,a) \models   \exists t_1,\cdots,t_n \ \forall  \bar{x} \ \exists \bar{y}\cdots \ \psi(\bar{x},\bar{y},a, t_1,\cdots,t_n). $$
  Taking into account the assumption $\tp_\Phi(a) \subseteq \tp_\Phi(b)$ it follows that:
     $$(M,b) \models   \exists t_1,\cdots,t_n \ \forall  \bar{x} \ \exists  \bar{y}\cdots  \ \psi(\bar{x},\bar{y},b, t_1,\cdots,t_n). $$
This, together with the fact that
$b\not\in p(M)$ implies that
 $(M,b) \models T^\prime$. 
 \par 
Now, compactness implies that there exists  a model $(M^*,b^*)$ for the theory $T^*$. 
As $M^*\models \diag_{\Pi_{n}}(M)$, 
it is easy to observe that $M$ is $\Sigma_n$-closed in $M^*$. By the assumption of the theorem, 
we have $p(M) \subseteq p(M^*)$. But this leads to a contradiction, because $a$ belongs to $p(M)$ but it is not in $p(M^*)$.

    \end{proof}

    \begin{cor}
    	\label{cor-pi-n+1-definition}
    	Let $T$ be a theory in $L\cup \{p\}$ where $p$ is a unary predicate.    	Then there is a 
    	$  \Pi_{n+1}$
    	formula that defines
    	$p$ 
    	in all models of $T$, if and only if:
    	\begin{quote}
    		For each pair of  models $(M,p(M))$ and $(M^*,p(M^*))$ of $T$ if $M$ is 	$\Sigma_n$-closed in $M^*$,  then
    		for all $a\in M$ if $M^*\models p(a)$ then $M\models p(a)$.
    	\end{quote}
    \end{cor}
\begin{proof}
	Again the `only-if' direction is easy to prove. For the `if' direction, 
apply the previous lemma to $p^c$, the complement of the predicate $p$.
\end{proof}

\section{The Ax-Kochen Theorem (an $ \exists\forall $-form)}
As in the introduction, in this section and the rest, the value group and residue field of a valued field $(K,A)$ are  denoted respectively by
$\Gamma$ and $k$. 
\label{ax-kochen}
The well-known theorem of Ax and Kochen on Henselian valued fields, asserts the following:
 \begin{quote}
  Suppose that Henselian valued fields $(K_1,A_1)$ and $(K_2,A_2)$ are  of \mbox{equicharacteristic $0$}. Then 
$(K_1,A_1) \equiv (K_2,A_2)$ as valued fields,  if and only if $\Gamma_1 \equiv \Gamma_2$ as ordered abelian groups and $k_1 \equiv k_2$ as fields.
 \end{quote}
Kuhlmann and Prestel  (in an appendix to \cite{kuh}) prove a version of the theorem above as follows:
   	 \begin{proposition}[Existential Ax-Kochen theorem]\label{kuhl-ax}
	Let
	$(K_1, A_1)$
	be a Henselian valued field with   residue characteristic zero, and
	$(K_2, A_2)$ be
	a valued field extension of
	$(K_1, A_1)$.
	If
	$k_1\subseteq_{_{\exists}}k_2$  and
	$\Gamma_1 \subseteq_{_{\exists}}\Gamma_2$, then $(K_1, A_1)\subseteq_{_{\exists}} (K_2, A_2)$.
\end{proposition}
 In this section we aim to rely on Kulhmann and Prestel's result to further demonstrate that under the same conditions:
 \begin{quote}
 $(K_1,A_1)$ is $\exists\forall$-closed in 
 $(K_2,A_2)$ as valued fields, if and only if 
 $\Gamma_1$ is $\exists\forall$-closed in $\Gamma_2$ as ordered abelian groups and 
 $k_1$ is $\exists\forall$-closed in $k_2$ as fields.
 \end{quote}
 Again the proof will be provided after recalling some facts from elementary  model theory. 
   	 \subsection{Elementary model theory facts}
   	 We start by fixing some notation:
   	 \begin{notation}
   	 	\label{notation-fMNdelta}
   	 Let $\Delta$ be a set of formulas in a language $L$,  $M$ and $N$ be two structures, and
   	 $f:M\to N$ be a map.
   	 \begin{itemize}
   	 	\item 
   	 	We use the notation
   	 	$f : M\stackrel{\Delta}{\longrightarrow } N$ whenever for all formulas
   	 	$\varphi\in \Delta$ and all $\bar{a}\in M$
   	 	if $M\models \varphi(\bar{a})$ then $N\models \varphi(f(\bar{a}))$.   	 
   	 	\item 
   	 	We use the notation
   	 $  {M}\overset{\Delta}{\Longrightarrow}   {N}$
   	 whenever for all \textit{sentences} $\varphi\in \Delta$, if ${M}\models \varphi$  then $N\models \varphi$.	
   	 \end{itemize}
   	
   	 \end{notation}
   	The two notions above are linked to one another via the following theorem.  
   	 \begin{thm}\cite[Chapter 3]{tent}
   	 	\label{tent-ziegler}
   	 	Assume that $\Delta$ is a set of formulas closed under $\wedge,\exists$. Then the following are equivalent:
   	 	\begin{enumerate}
   	 		\item 
   	 		$ {M}\overset{\Delta}{\Longrightarrow} {N}$.
   	 		\item 
   	 		There exist a structure ${N}'\equiv {N}$ and a mapping $f:M\overset{\Delta}{\longrightarrow} N'$.
   	 	\end{enumerate}
   	 \end{thm}
   	 We can apply  Notation \ref{notation-fMNdelta} in a slightly different way:
   	 Let $\Delta$ be the set of formulas of certain \textit{complexity}, for example 
   	 the set of existential or $\exists\forall$-formulas.
   	 When we say that $f:M\overset{\Delta}{\longrightarrow} N$, or
   	 $M\overset{\Delta}{\Longrightarrow} N$,
   	 \textit{in the language
   	 $L$},  we mean that the condition in Notation \ref{notation-fMNdelta}
   	 holds for formulas of the form $\Delta$ in the language $L$.
   	Now according to Definition \ref{defn-sigma-n-closed},
   	 it is easy to observe that
   	 	${M}{\subseteq}_{_{\Delta}} {N}$ in the language $L$ if and only if $ {N}\overset{\Delta}{\Longrightarrow} {M}$ in the language $L(M)$.
   	  Combining  this observation with Theorem \ref{tent-ziegler} we have the following:
   	 \begin{cor}
   	 	\label{coro2}
   	 	Let $M$ and $N$ be two structures. Then
   	 	${M}{\subseteq_{_{\Delta}}} {N}$
   	 	 if and only if there exist a structure ${M}\prec {M}'$ and a mapping
   	 	 $f:N\overset{\Delta}{\longrightarrow} M'$. Using compactness, this is equivalent to the fact that there exists
   	 	  $f:N\overset{\Delta}{\longrightarrow} M'$ 
 for any $|N|$-saturated elementary extension $M'$ of $M$.
   	 \end{cor}
The corollary above combined with the following easy observation
   	 provide us with useful consequences. 
   	 \begin{obs}
   	 	\label{obs-existsforall=exists}
   	 	For any embedding $f:N\longrightarrow M'$ we have 	$f:N\overset{\exists}{\longrightarrow} M'$.  Also we have 
   	 	$f:N\overset{\forall }{\longrightarrow} M'$  if and only if $f:N\overset{\exists\forall}{\longrightarrow} M'$.
   	 \end{obs}
   	 We will need the second item in the corollary below, in the sections to come. 
   	 \begin{cor}\label{ex-closed}\hfill
   	 	\begin{enumerate}
   	 		\item 
   	 		$ M{\subseteq_{_{\exists}}} {N} $ if and only if there exists
   	 		an embedding
   	 		$f:N\longrightarrow M'$ for 
   	 		 a structure $M'$ with ${M}\prec {M}'$.
   	 		 In other words, 
   	 		 ${M}{\subseteq_{_{\exists}}} {N}$ if and only if for all $|N|$-saturated elementary extensions ${M}'$  of ${M}$ there is an embedding 
   	 		 $f:{N} \longrightarrow {M}'$, that is such that
   	 		 $f(N)\subseteq {M}'$.
   	 		\item 
   	 		$ {M}{\subseteq_{_{\exists\forall}}} {N}  $ if and only if there  exists 
   	 		a mapping
   	 		$f:N\overset{\forall}{\longrightarrow} M'$ for a
   	 		structure 
   	 		$M'$ with
   	 		${M}\prec {M}'$.
   	 		In other words, 
   	 		${M}{\subseteq_{_{\exists\forall}}} {N}$ if and only if for all $|N|$-saturated elementary extensions ${M}'$  of ${M}$ there is 
   	 		a mapping $f$ such that 
   	 		$f:{N} \overset{\forall}{\longrightarrow}{M}'$, that is such that
   	 		$f(N){\subseteq_{_{\exists}}} {M}'$.
   	 	\end{enumerate}
   	 \end{cor}
   	  \subsection{Relation to the Ax-Kochen Theorem}
Recalling Proposition \ref{kuhl-ax}, 
   	   	 by the machinery developed in  previous sections, we have an easy generalization as below.
   	   	  Note that
   	   	  in the following proof, we use both the proof and the statement of theorem Ax-Kochen-Ershov in \cite[page 183]{kuh}. 
   	 \begin{thm} 
   	 	\label{ax- asli}
   		Let
   		$(K_1, A_1)$
   		be a Henselian valued field  with residue characteristic zero, and
   		$(K_2, A_2)$ be
   		a valued field extension of
   		$(K_1, A_1)$. If 
   		$k_1\subseteq_{_{\exists\forall }} k_2$  and
   		$\Gamma_1 \subseteq_{_{\exists\forall }}\Gamma_2$, then $(K_1, A_1)\subseteq_{_{\exists\forall }} (K_2, A_2)$.
   	\end{thm}
   	 	\begin{proof}
   	 		Let $(K^*, A^*)$ be a $|K_2|^+$-saturated 
   	 		elementary extension of $(K_1, A_1)$.
   	 		 Since the extension is
   	 		elementary, $(K^*, A^*)$ is a Henselian valued field. Moreover, the value group $\Gamma^*$ is an
   	 		elementary $|K_2|^+$-saturated 
   	 		extension of $\Gamma_1$.
   	 		As
   	 		$\Gamma_1 \subseteq_{_{\exists\forall }} \Gamma_2$, by Corollary \ref{ex-closed}
   	 		there exists a mapping $f$ such that
   	 		$f:\Gamma_2 \overset{\forall}{\longrightarrow} \Gamma^*$, therefore 
   	 		$ f(\Gamma_2) \subseteq_{_{\exists}} \Gamma^* $. 
   	 		Similarly, there exists a mapping $g$ such that
   	 		$g (k_2)  \subseteq_{_{\exists}} k^* $. 
   	 		Following steps 1, 2 and 3 of the proof on page 192 of \cite{kuh}, it is possible to find
   	 		 a valued field $ (K_2',A'_2) $ isomorphic to $(K_2,A_2)$ whose residue field and value group
   	 		 are respectively $g(k_2)$ and $f(\Gamma_2)$. So $(K'_2,A'_2)$ is an extension of $(K_1,A_1)$
   	 		 and its residue field and value group are respectively $\exists$-closed in the residue field and value group
   	 		 of $(K^*,A^*)$.
   	 		 By  Proposition  \ref{kuhl-ax} we have $(K'_2,A'_2)$ is $\exists$-closed in $(K^*,A^*)$, thus we have a mapping $h$ such that
   	 		$ h((	K_2,A_2))  \subseteq_{_{\exists}} (K^*,A^*) $, equivalently
   	 		such that  $h:(K_2,A_2)\overset{\forall}{\longrightarrow} (K^*,A^*)$.
   	 		Now
by Corollary
   	 		\ref{ex-closed}, we conclude that
   	 		$ (	K_1,A_1) \subseteq_{_{\exists\forall}} (K_2,A_2)   $.
    	    \end{proof}

 	\section{Finding suitable groups}\label{group-construction}
 \subsection{Quantifier-Elimination for the theory of ordered abelian groups}
  Let $L_{\text{oag}}$ denote  the language of ordered abelian groups (containing the binary function symbol ``$-$'' as well).
  By a \textit{congruence equation modulo $n$}, we mean 
  a formula of the form\mbox{ $\exists t\ y-x=nt$}, which we denote  by  ``$x\overset{n}{\equiv} y$''. 
  Let $ \psi $ be a conjunction of congruence equations, and
  $\varphi(\bar{z}_1,\ldots,\bar{z}_n;\bar{y})$ be  a formula of the following form:
\begin{align}
\label{r-phi-exists-u}
  \exists u_1,\cdots,u_r\quad \psi(\bar{z}_1,\cdots,\bar{z}_n;\bar{y},u_1,\cdots,u_r).
\end{align}
    For $\varphi$ as above, let
  $R_{\varphi}(x_1,\ldots,x_n;\bar{y})$ be the following formula:
 \begin{align}
 \label{general-form-of-R-phi}
  \exists \bar{z}_1,\cdots,\bar{z}_n \quad \Big(\bigwedge_{i=1}^n( 0<\bar{z}_i<x_i)\wedge \varphi(\bar{z}_1,\cdots,\bar{z}_n;\bar{y})\Big),
 \end{align}
  where by $\bar{z}_i<x$ we mean $z_{i_j}<x$ for all $i_j$ in the representation $\bar{z}_i=(z_{i_1},\ldots,z_{i_k})$.
  Let $L''_{\text{oag}}$~\footnote{$L''$ is used to be compatible with the notation of \cite{group}.}
   be the language obtained by adding to $L_{\text{oag}}$,  predicates  
   $\overset{n}{\equiv}$
    for congruence equations for all
  $n \in \mathbb{N}$,
   and $R_{\varphi}$ for all 
  $\varphi(\bar{z}_1,\ldots,\bar{z}_n;\bar{y})$ of the mentioned form. 
 In \cite{group} the following result is established for abelian groups in the mentioned language:
  \begin{fact}\cite[Theorem 1.1]{group}
  	\label{1va1 group}
  	Let $G,H$ and $N$ be ordered abelian groups such that $N$ is a subgroup of	$G$ and $H$. Then the following are equivalent:
  	\begin{enumerate}
  		\item  $G$ and $H$
  		satisfy the same existential sentences in $L_{\text{oag}}(N)$,
  		\item $G,H$ satisfy the same atomic sentences in  the language $L''_{\text{oag}}(N)$. 
  	\end{enumerate}
  \end{fact}
The fact above can be put as a  quantifier-elimination statement as below:
  \begin{thm}
Each existential $L_{\text{oag}}$-formula is equivalent to a conjunction of disjunctions of atomic formulas in $L''_{\text{oag}}$.
  \end{thm}
\begin{proof}
	Let $\varphi(\bar{x})$ be an existential $L_{\text{oag}}$-formula. Treat $\varphi(\bar{x})$ as a predicate $p(\bar{x})$ in the language $L_{\text{oag}}\cup \{p\}$.
	Let $\Phi$ be the set of all (conjunction of disjunctions of) atomic formulas in $L''_{\text{oag}}$.  According to Lemma \ref{lem1 new}, to prove the theorem, it suffices to show that for all ordered abelian groups $M$ we have:
\begin{quote}
	For each $\bar{a},\bar{b}\in M$ if $\tp_{\Phi}(\bar{a})\subseteq \tp_{\Phi}(\bar{b})$ then if $\bar{a}\in \varphi(M^{|\bar{a}|})$ then $\bar{b}\in \varphi(M^{|\bar{a}|})$.
\end{quote}
Add a tuple of constants $\bar{c}$ to $L_{\text{oag}}$ and consider the $L_{\text{oag}}\cup \{\bar{c}\}$-structures $(M,\bar{a})$ and $(M,\bar{b})$.
Consider the group generated by $\bar{c}$ a common substructure of $(M,\bar{a})$ and $(M,\bar{b})$. 
By item 2 we have the condition  ``$\tp_{\Phi}(\bar{a})\subseteq \tp_{\Phi}(\bar{b})$'' in Fact \ref{1va1 group}. By item 1 in the mentioned fact,
 for $\bar{a}$ and $\bar{b}$  we have
 $M\models \varphi(\bar{a})$ if and only if $M\models \varphi(\bar{b})$.
 \end{proof}
  \begin{cor}\label{exist-forall}
  	Each $\exists\forall$-formula in $L_{\text{oag}}$ has an equivalent of the form:
  	\[
  	\exists x_1,\cdots,x_n \bigvee_k \bigwedge_t \psi_{k,t} 
  	\]
  	where each $\psi_{k,t}$ is the  negation of an $L''_{\text{oag}}$-atomic formula.
  \end{cor}
  \begin{proof}
Replace  $\forall$  with $\neg \exists \neg$, and use the above theorem on existential $L_{\text{oag}}$ formulas. 
  \end{proof}
  \subsection{Recalling groups constructed in \cite{main}}
  \label{Recalling-the-construction-in}
  As our construction is obtained by alteration of the example already given in
  \cite{main}, we find it necessary to first present their construction and investigate the reason it does not fit to our purpose. 
  \par 
  By $\mathbb{Z}_{(n)}$  denote the additive subgroup of the rational numbers consisting of elements
  $\frac{a}{b}$ where $b$ is not divisible by $n$.
  Representing $\mathbb{Z}_{(3)}$ with an square and $\mathbb{Z}_{(2)}$ with a circle,
   we let $\Gamma_1$ be the following ``direct sum":
  \[
  \square\square\square\cdots  \circ \square\square\square\cdots \circ \square\square\square\cdots \circ \cdots
  \]
  and
  $\Gamma_2$ be the following direct sum:
  \[
  \cdots \circ \square \circ \square\circ \square 
  \]
  and consider the group $\group$ with the lexicographic order as follows:
  \[
  \overbrace{{\cdots \circ \square \circ \square\circ \square}}^{\Gamma_2} |\overbrace{{\square}\square\square \cdots  {\circ} \square\square\square \cdots \circ \square\square\square\cdots \circ \cdots}^{\Gamma_1}
  \]
 In \cite[page 364]{main} 
two maps $f_1,f_2:\group\to \group$ are 
explicitly
introduced. 
We refer the reader to the definitions of the maps in there, and suffice here to visually explain them. The  function $f_1$ maps the first square of $\Gamma_1$
to the last (circle, square) of $\Gamma_2$---that is in $f_1(\group)$ the mentioned circle contains only zero---and shifts all shapes as below:
	\[
\xymatrix@C=.5mm@R=2mm{
	\cdots & \circ & \square & \circ\ar[dll] & \square\ar[dll] & \circ \ar[dll]&
	\square\ar[dll] & | & 
	\blacksquare \ar[dll]  & 
	\square \ar[dl]& \square\ar[dl]& \cdots & \circ\ar[d] & \square \ar[d]& \square \ar[d]& 
	\cdots & \circ \ar[d]& \cdots \\
	\cdots & \circ & \square & \circ & \square & \scalebox{1.8}{$\bullet$} &
	\blacksquare & | & 
	\square  & 
	\square &\square&  \cdots & \circ & \square & \square & 
	\cdots & \circ & \cdots 
}
\]
Also let $f_2$ be the following map which avoids a block of squares in its image:
	\[
\xymatrix@C=1mm@R=5mm{
	\cdots & \circ & \square & \circ\ar[drr] & \square\ar[drr] & \scalebox{1.8}{$\bullet$} \ar[drrrrrr]&
	\blacksquare\ar[drrrrrr] & | & 
	\square\ar[drrrrr]   & 
	\square &  \cdots & \circ & \square & \square & 
	\cdots & \circ & \cdots \\
	\cdots & \circ & \square & \circ & \square & \circ &
	\square & | & 
	\blacksquare  & 
	\blacksquare &  \cdots & \scalebox{1.8}{$\bullet$} & \blacksquare & \square & 
	\cdots & \circ & \cdots 
}
\]
It is proved in \cite{main}  that firstly, 
 $\Gamma_1$ is definable in $\Gamma_2\oplus \Gamma_1$, and secondly,
 $f_i$ are $\exists$-embeddings; meaning that
$f_i(\Gamma_2\oplus \Gamma_1)$ are $\exists$-closed in $\Gamma_2\oplus \Gamma_1$.
In the following, we first explain the reasons behind the mentioned facts, 
and then find a replacement for this example to suit our purpose.
\begin{notation}
	\label{I-a}
 Let $ G=\oplus_{i\in I} G_i $ be a direct sum of ordered abelian groups where $I$ is a linearly ordered index set.
	 We denote each element of $G$ by a boldface letter
	$\mathbf{a}$ where $\mathbf{a}=(a_i)_{i\in I}$. For such an element $\bf{a}$ by
	$I(\bf{a})$ we mean the smallest index $i$ such that $a_i\neq 0$.
	\end{notation}
	\begin{defn}
Let $G$ be an ordered abelian group and $a$ an element in $G$ not divisible by $n$. 
Define $H_a^n$ as follows:
\[
H^n_a=\{b: \forall y   \big(0<y<b\to y\not \stackrel{n}{\equiv} a\big) \}. 
\]		
	\end{defn}

   \begin{obs}
	\label{obs-psi-a-b}
	In a direct sum $G=\oplus_{i\in I} G_i$  with each $G_i$  an ordered abelian group and $I$ a linear order,
	the formula 
	\begin{align}
		\psi_n(\mathbf{a},\mathbf{b}):=    ``\bf{b}\in H^n_{\bf{a}}"
		\label{psi-a-b}
	\end{align}
	is equivalent to the fact that: 
	\begin{enumerate}
		\item $I(\bf{a})<I(\bf{b})$ and there is an element $a_j$ with $j<I(\bf{b})$ such that
		$a_j$ is not divisible by $n$, or
		\item $I(\bf{a})=I(\bf{b})=i$ and 
		$G_{i}\models \psi_n(a_i,b_i)$.
	\end{enumerate}
\end{obs}
\begin{cor}
In $\group$, 
the formula
``$ \bf{b}\in H^2_{\bf{a}}\cup H^3_{\bf{a}} $", that is the formula $\psi_2(\bf{a},\bf{b})\vee \psi_3(\bf{a},\bf{b})$ implies $I(\bf{a})<I(\bf{b})$.
	\end{cor}
	Although ``$ \bf{b}\in H^2_{\bf{a}}\cup H^3_{\bf{a}} $" is not equivalent to 
	$I(\bf{a})<I(\bf{b})$, we still use the notation
``$I(\bf{a})<I(\bf{b})$" to refer to the formula $\psi_2(\bf{a},\bf{b})\vee \psi_3(\bf{a},\bf{b})$. It is clear that
``$I(\bf{a})<I(\bf{b})$" is a $\forall$-formula.
In the following corollary, we only consider the map $f_1$ but the statement holds also for $f_2$
for a very similar reason. 
\begin{cor}\hfill
	\label{f-baste-vojodi}
	\begin{enumerate}
		\item $ f_1(\group) $ is $ \exists$-closed in $\group$. 
		\item $\Gamma_1$ is a definable subset of $ \group $ by a $\exists\forall\exists\forall$-formula in $L_{\text{oag}}$.
		\item $ f_1(\group) $ is \textit{not} $ \exists\forall $-closed in $ \group $.

	\end{enumerate}
\end{cor}
\begin{proof}While items 1 and 2 are already proved in \cite{main}, we include their proofs here to improve the readability of the rest of the article based on our notation. 
	For 1. note that the group $ \mathbb{Z}_{(3)} $ is a convex subgroup of $ \mathbb{Z}_{(2)}\oplus \mathbb{Z}_{(3)} $, hence $ \mathbb{Z}_{(3)} $ is $\exists$-closed
			in $ \mathbb{Z}_{(2)}\oplus \mathbb{Z}_{(3)} $ (Corollary 1.4, \cite{group}).
					It is known that whenever groups $ G $ and $ H $ are respectively $\exists$-closed in $ G' $ and $ H' $, then 
			$ G\oplus H $ 
			is also $\exists$-closed in $ G'\oplus H' $ (Corollary 1.7, \cite{group}). Hence the map $f_1$ is an $\exists$-embedding.
\par 
	For item 2, note that
		if $\bf{a},\bf{b}$ are elements with the following properties, then $a_{I(\bf{a})}$
		is in the last square of $\Gamma_2$ and $b_{I(\bf{b})}$ is in the first square of $\Gamma_1$:
		\begin{itemize}
			\item $ \bf{a}  $ and $\bf{b}$ are both divisible by 2 and not divisible by 3,
			\item ``$I(\bf{a})<I(\bf{b})$'',
			\item there is no $\bf{c}$ such that ``$I(\bf{a})<I(\bf{c})<I(\bf{b})$",
			\item ``$I(\bf{a})$ is minimal with these properties".
		\end{itemize}
	Then $H^3_{\bf{a}}=\Gamma_1$. The fact that the formula that defines $\Gamma_1$ is $\exists\forall\exists\forall$ is an easy check.
\par 
For item 3,
		denote  by $\mathbf{a}$ the element $(a_i)_{i\in I}$ with $a_j=a>0$ for
	the index $j$ of the critical circle, that is the circle which is zero in the image of $f_1$, and
	$a_i=0$ elsewhere.
	Assume further  that $a$, and hence $ \mathbf{a} $,
	is not divisible by 2.
	Also
	let $\mathbf{b}$ be an element in $ f(\Gamma_2 \oplus \Gamma_1)$
	such that  $b_{j+1}=b>0$ and the rest of the coordinates are zero.
	Let $\mathbf{c}$ be an element such that  $c_{j-1}=c>0$, and the rest of the coordinates are zero.  Assume further that $b$ and $c$ are not divisible by 3.
	Then $\mathbf{a}$ serves as a witness for the existential quantifier in the $\exists\forall$-formula below:
	     	\[
	\exists x\quad ``I(\bf{c})<I(x)<I(\bf{b})".
	\]
Obviously any witness to existential quantifier above, needs to be non-zero in the critical circle.
\end{proof}
\subsection{Modifying the Construction}
\label{key}
Let the squares be the group $\mathbb{Z}[c_1,c_2,\ldots]=
 \mathbb{Z}\oplus\bigoplus_{i\in \mathbb{N}} \mathbb{Z}c_i$, that is the group generated by
$ \mathbb{Z} $ together with
a countable set of elements $ c_i $.
We further impose the condition that
the $c_i$'s 
are infinitely small transcendental elements with $c_1>c_2>\ldots$
belonging to different archimedian  classes;
that is $nc_{i+1}<c_{i}$ for all  $i$ and all $n\in \mathbb{N}$. We also put the lexicographic order on each square.
Also let the circles be $\mathbb{Q}$.
Let $\Lambda_1,\Lambda_2$ be obtained similar to $\Gamma_1,\Gamma_2$ in Section \ref{Recalling-the-construction-in} and consider the group $\Group$.
 \begin{notation}
 	 Let $I(\bf{a})$ for an element $\bf{a}\in \Group$ be defined as in Notation \ref{I-a}.
 	 If $I(\bf{a})=i$ and $a_{i}$ is in a square, then let $J(\bf{a})$ be the smallest index $j$ such that $a_{ij}\neq 0$.
 	 Now we define
 	$I_n(\bf{a})$ to be the smallest index $i$ such that $a_i$ is not divisible by $ n $, when such index exists. Of course in this case $a_i$ belongs to a square.
 	Let $I_n(\bf{a})=i$ and $\bf{a}_i=(a_{ij})_j$. Define $J_n(\bf{a})$ to be the smallest index $j$ such that $a_{ij}$ is not divisible by $ n $, when such index exists.
 	For an element $\bf{a}\in \Group$ by $B(\bf{a})$ we mean the set of elements
 	$\bf{t}$ such that $\bf{0}<\bf{t}<\bf{a}$.
 \end{notation} 
 
 \begin{obs}In group $\Group$,
 	the formula $\psi_n(\mathbf{a},\mathbf{b})$ in \ref{obs-psi-a-b}  means that $(I_n(\bf{a}),J_n(\bf{a}))<(I_n(\bf{b}),J_n(\bf{b}))$ with the lexicographic
 	order.
 \end{obs}
 In the following by $H_{{a}}$ we mean $H^2_{a}$ as previously defined, and we also define
 \[
 H'_{{a}}=\left\{{t}: \exists {t}' \Big({0}<{t}'<{a}\wedge \forall z \big({0}<{z}<{t}\to{z}\not \Equiv{2} {t}' \big)\Big)
 \right\}.
 \]
 It is obvious that 
 \[ 
 H'_{{a}}=\bigcup_{{0}<{b}<{a}} H_{{b}}.
 \]
 For our group $\Group=\oplus_{i\in \mathbb{Z}} G_i$ where each $G_i$ is either circle or square,  $H_{\bf{a}}$ and
 $H'_{\bf{a}}$ can be determined as follows:
 \begin{enumerate}
 	
 	\item If $a_{I(\bf{a})}$ is in a circle then 
 	$H_{\bf{a}}=\emptyset$ if $\bf{a}$ is divisible by 2, and
 	$H_{\bf{a}}=\Big(\bigoplus_{j>0}\mathbb{Z}c_{J_2(\bf{a})+j}\Big) \oplus\bigoplus_{i>I_2(\bf{a})}G_i$ otherwise. In any case, $H'_{\bf{a}}=\Big(\bigoplus_{i>0}\mathbb{Z}c_i\Big)\oplus \bigoplus_{i>I(\bf{a})+1} G_i$.
 	\item If $a_{I(\bf{a})}$ is in a square, and $\bf{a}$ is divisible by 2 then $H_{\bf{a}}=\emptyset$.
 	If $\bf{a}$ is not divisible by 2 then $H_{\bf{a}}=\Big(\bigoplus_{j>0} \mathbb{Z}c_{J_2(\bf{a})+j} \Big)\oplus  \bigoplus_{i>I_2(\bf{a})}G_i$.
 	In any case,  
 	$H'_{\bf{a}}=\Big(\bigoplus_{i>J(\bf{a})}\mathbb{Z}c_i\Big)\oplus \bigoplus_{i>I(\bf{a})} G_i$.
 \end{enumerate}
 \begin{remark}
 	It is important to notice that $H'_{\bf{a}}$ only depends on $a_{I(\bf{a})}$.
Also, based on the above, 	we leave it to the reader to verify that the situation in
 	the proof of item 3 of Corollary \ref{f-baste-vojodi} does not occur in this case.
 \end{remark}
 \begin{thm}
 	$\Lambda_1$ is definable in $\Group$ without parameters in $L_{\rm oag}$.
 \end{thm}
 \begin{proof}
 	In the following we will write formulas with variables $(x,y)$ such that
 	if $(\bf{x},\bf{y})\in \Group$ satisfies these formulas, then $y_{I(\bf{y})}$ and $x_{I(\bf{x})}$ are   respectively in the last square of $\Lambda_2$ and 
 	first square of $\Lambda_1$.
 	First note  that if an element $\bf{t}\in \Group$ satisfies the following formula
 	$\psi(t)$, then ${t}_{I(\bf{t})}$ is in a square:
 	\begin{align*}
\exists z \left(z\in B(t) \wedge	 \forall z' \left( z'\in H'_t \to z\overset{2}{\not\equiv} z'\right)	\right).
 	\end{align*}
 	So our first condition is:
 	\begin{align}
 		\label{this}
 	\psi(x)\wedge \psi(y).
 	\end{align}
 	Now consider the following condition:
 	\begin{align}
 		\label{that} 
 		H'(x)\subsetneq H'(y) \wedge
 	\forall  z \left(z\in H'_y-(H'_x\cup  \{t:H'_t=H'_x\}) \to  \exists z' ( H'_x\subsetneq H'_{z'}\subsetneq H'_z)\right).	
 	\end{align}
 	If a tuple $(\bf{x},\bf{y})\in \Group$ satisfies 
 	conditions \eqref{this} and \eqref{that}, then  both $\bf{x},\bf{y}$ start in squares, the square for $\bf{x}$ comes after the square for $\bf{y}$, and ${x}_{I(\bf{x})}$ is in $\mathbb{Z}$, that is the beginning of the square.
 	\par 
 	Finally we add the conditions that both $H'_y$ and $H'_x$ are maximal with respect to the mentioned conditions.
 	Then $\Lambda_1$ is definable by being equal to $H'_{x}\cup \{t:H'_{{t}}=H'_{{x}}\}$ for all $(x,y)$
 	satisfying the conditions above.
 	 \end{proof}
 	 Let $f_1$ and $f_2$ be the maps introduced in the previous section, of course adjusted for our new
 	 squares and circles.
 	  Clearly, $f_i(\Group) $ for $i=1,2$ is $\exists$-closed in $\Group$. We aim to prove that
 	 $f_i(\Group)$ for $i=1,2$ is indeed $\exists\forall$-closed in $\Group$. 
 	 \par 
In observation \ref{obs-psi-a-b}
we interpreted the meaning of the formula $\psi(\bf{a},\bf{b})$---which is already an instance of a formula $\neg R_{\varphi}$.
In the following we investigate the meaning of a formula
$\neg R_{\varphi}$, for  a formula $R_{\varphi}$ as in  Equation \eqref{general-form-of-R-phi} in its general form. 
The following example would be beneficial as it treats a special case:
\begin{example} In $ \Group $
	the  formula:
	\[
	\neg \Big(\exists z_1, z_2 \ ( 0<z_1<\bf{a}_1  \wedge  0<z_2<\bf{a}_2  \wedge 
	z_1\overset{2}{\equiv} \bf{b}_1 \wedge z_2\overset{2}{\equiv} \bf{b}_2 \wedge z_1\overset{2}{\equiv} z_2)\Big)
	\]
	is equivalent to one of the following:
	\begin{enumerate}
		\item $ (I_2(\bf{b}_1),J_2(\bf{b}_1))<(I_2(\bf{a}_1),J_2(\bf{a}_1)) $,
		\item $ (I_2(\bf{b}_2),J_2(\bf{b}_2))<(I_2(\bf{a}_2),J_2(\bf{a}_2)) $,
		\item $\bf{b}_1\not \overset{2}{\equiv} \bf{b}_2$.
	\end{enumerate}
\end{example}

\begin{cor}
	\label{cor-general-neg-Rphi}
	In $ \Group $,
	the general case of a formula of the form
	$\neg R_{\varphi}(\bf{a}_1,\ldots,\bf{a}_n;\bf{b}_1,\ldots,\bf{b}_m)$ for $\varphi$ as in Equation \eqref{general-form-of-R-phi}
	describes instances of $(I_n(\bf{b}_i),J_n(\bf{b}_i))<(I_n(\bf{a}_j),J_n(\bf{a}_j))$, or $\bf{b}_i\not \overset{n}{\equiv} \bf{b}_j$.
\end{cor}
\begin{obs}\hfill
	\label{index-lemma}
	Let $ \bf{a}\in \Group $ be given such that $(I_n(\bf{a}),J_n(\bf{a}))$ exists. Then there are $\bf{c},\bf{d}\in f_1(\Group)$ such that
	whenever $\bf{c}<\bf{u}<\bf{d}$, we have $(I_
	n(\bf{u}),J_n(\bf{u}))=(I_n(\bf{a}),J_n(\bf{a}))$.
\end{obs}
\begin{cor}
\label{crucial-natije-r-phi}
	In $\Group$
every formula $\neg R_{\varphi}$
	can be written as a conjunction of non-congruence equations and inequalities like 
	$\mathbf{c}<x<\mathbf{d}$,
	 where
	  $\mathbf{c},\mathbf{d}\in f_1(\Group)$.
\end{cor}
\begin{remark}
	\hfill
	\label{remark-essential}
	\begin{enumerate}
		\item 
		It is crucial to observe that formulas
		like $I_n(\bf{a})<I_n(\bf{x})<I_n(\bf{b})$ do not force
		$\bf{x}$ to have entries in the critical circle (or the critical rectangle).
		\item 
		Assume that $\bf{t}$ and $\bf{\epsilon}>0$ in $\Group$ are given. 
		Then it is possible to find $\bf{t}'\in f_1(\Group) $ such that $|\bf{t}'-\bf{t}|<\bf{\epsilon}$ and $\bf{t}\not\stackrel{n}{\equiv} \bf{t}'$.
		\label{epsilon}
	\end{enumerate}
\end{remark}

\begin{thm}
	\label{exists-forall-closed-f1}
	$ f_1(\Group)
	$
	is
	$ \exists\forall $-closed in 
	$ \Group $ in 
	$L_{\text{oag}}$.
\end{thm}
\begin{proof}
	Let $\phi(\bar{x},\bar{y},\bar{z})$ be a quantifier-free $ L_{\text{oag}}$-formula. We have to show that if the formula $  \exists\bar{x}\ \forall \bar{y}\ \phi(\bar{x},\bar{y},\bar{\bf{a}}) $ 
	with $ \bar{\bf{a}}\in  (f_1(\Group)) ^{| \bar{\bf{a}}|} $
	is satisfied in $\Group$, then it is also 
	satisfied in $f_1(\Group)$.
	By Lemma \ref{exist-forall},  the formula 
	$ \exists\bar{x}\ \forall \bar{y}\ \phi(\bar{x},\bar{y},\bar{\bf{a}}) $ is equivalent to $ \exists \bar{x} \ {\bigvee}_{k}{\bigwedge}_{t} \psi_{k,t} (\bar{x},\bar{\bf{a}})$,  where each $ \psi_{k,t} $ is the negation of an atomic formula in  $ L''_{\text{oag}} $. We need only to show that for each $ k $, if the formula $\chi $ of the form $\exists \bar{x} \  {\bigwedge}_t \psi_{k,t}(\bar{x},\bar{\bf{a}})$ holds in  $\Group$, then it holds in $f_1(\Group)$.
	Note that  $ \psi_{k,t} (\bar{x},\bar{\bf{a}}) $ can be one of the following forms:
	\begin{align}
		\label{r-phi}
		&\neg R_{\varphi},\text{ for some $\varphi$  of the form \eqref{r-phi-exists-u}}, \\
		&\sum m_ix_i\leq \bf{a}.\label{eq-sum-leq}\\
		&\sum m_ix_i\not\stackrel{2}{\equiv} \bf{a}, \label{eq-sum-not-equiv}
	\end{align}
 
	Therefore, the general form of $\chi$ is as follows:
	\begin{align}
		\label{general-form-of-formula}
		\exists x_1, \cdots, x_m \  \big(\bigwedge_{\varphi \in S}\neg R_{\varphi} \wedge \bigwedge_{\theta\in \varTheta_1}\theta \wedge\bigwedge_{\theta\in \varTheta_2} \theta \big)\end{align}
	where 
	$S$ is a finite set of formulas of form \eqref{r-phi},
	$\varTheta_1$ and $\varTheta_2$  are finite sets 
	consisting, respectively, of formulas of the form 	\eqref{eq-sum-leq}, and \eqref{eq-sum-not-equiv}.
Now assume that $\bf{b}_1,\ldots,\bf{b}_m$ are elements in 
$\Group$, that witness the existential quantifier in 
formula \eqref{general-form-of-formula}. By Corollary~\ref{crucial-natije-r-phi}, each $\neg R_{\varphi}$ can be rewritten as formulas in $\varTheta_1$ and $\varTheta_2$, so it suffices to consider only formulas in $\varTheta_1$ and $\varTheta_2$.

The group
	$ f_1(\Group) $
	is $\exists$-closed in 
	$\Group $, hence 
	there exist elements
	$\bf{c}_1,\ldots,\bf{c}_m$  in $f_1(\Group)$ that satisfy 
	$\bigwedge_{\theta\in \varTheta_1} \theta$.

	We claim that it is possible to alter each $\bf{c}_1,\ldots, \bf{c}_m$ so as to also satisfy
	$\bigwedge_{\theta\in \varTheta_2}\theta$.
	First  we find $\epsilon\in f_1(\Group)$ 
	such that for each inequality  $\sum m_ix_i \leq \bf{a}$ in $\varTheta_1$, we have
	$\sum m_ix_i \leq \bf{a}-\epsilon$.  Then we change $\bf{c}_i$'s to $\bf{c}'_i$'s using item
	\ref{epsilon} in the remark above,
	such that 
	$\sum m_i\bf{c}'_i \leq \bf{a}-\epsilon$ for each such inequality, and at the same time $\bf{c}'_i$'s satisfy the
	non-congruence equations involved in $\varTheta_2$. 
	\par 
	Finally note that we may assume that no equality sign has appeared in formulas in $\varTheta_1$, since each equality sign leads to writing some variables in terms of the others and reducing the number of variables. 
\end{proof}
\begin{thm}
	$ f_2(\Group)
	$
	is
	$ \exists\forall $-closed in 
	$ \Group $ in 
	$L_{\text{oag}}$.
\end{thm}
\begin{proof}
For any \(\mathbf{a}, \mathbf{b} \in f_2(\Group)\), whenever the inequality
$ I_n(\mathbf{a}) < I_n(x) < I_n(\mathbf{b})$
has a solution in $\Group $, it also has one inside $f_2(\Group)$. This means, just like in the proof of Theorem 4.18, we can replace all formulas of the form $\neg R_{\varphi}$ with appropriate formulas of the types \eqref{eq-sum-leq} and \eqref{eq-sum-not-equiv}. From this point onward, the proof proceeds by essentially the same argument employed in \Cref{exists-forall-closed-f1}.

\end{proof}

   \section{Final step}
   \label{final-step}
   Throughout this section, we let $k$  be a pseudo algebraically closed (PAC) field 
   whose algebraic part is not algebraically closed. For an example of such a field
   see \cite{main}.
   Also let
   $\Lambda_1$ and $ \Lambda_2$ be the ordered abelian groups introduced in the previous section. Also consider the field of Hahn Series
   $\mathcal{K} = k((\Lambda_1))((\Lambda_2))$, equipped with the
   $t$-adic valuation, whose value group, residue field, and valuation ring, respectively, are
   $\Lambda_2$,
   $k((\Lambda_1))$, and
   $A= k((\Lambda_1))[[\Lambda_2]]$. As our main result, in Subsection  \ref{Definability-of-A} we show that $A$ is definable without parameters and in Subsection
   \ref{The-impossibility-of-an-exists-forall-exists-definition-for-A}
    we show that it is not possible
   to define $A$ with an $\exists\forall\exists$-formula. The proof of the fact that
   $A$ is also not $\forall\exists\forall$-definable is similar.
   \subsection{Definability of $A$}
   \label{Definability-of-A}
   First note that $\mathcal{K}=k((\Lambda_1))((\Lambda_2))$ is isomorphic to the valued field
   $k((\Group))$.  By Theorem~2.1 in \cite{fehm}, the valuation ring $k[[\Group]]$ is $\exists$-definable in $\mathcal{K}$.
  It is straightforward to verify that
   $
   A = k[[\Group]] \cup k((\Lambda_1))$.
   Therefore, to show the definability of $A$, it is enough to show that $k((\Lambda_1))$ is a definable subset of $\mathcal{K}$ in ${L_\text{ring}}$.
\par 
  As proved in the previous section, the subgroup $\Lambda_1$ is definable in $ \Lambda_2 \oplus \Lambda_1$ in    $L_{\text{oag}}$.
   Our goal is to translate this definability into the $L_{\text{ring}}$-definability of $k((\Lambda_1))$ as a subset of $k((\Group))$. 
    This translation works as explained below.
    \par 
    Note that $k((\Lambda_1))=\{f\in k((\Group)): v(f)\in \Lambda_1\}$. The relation $v(f)\in \Lambda_1$ is definable
    in $\Group$ by an $L_{\text{oag}}$-formula. Atomic sub-formulas of such a formula  are instances of $v(f_1)<v(f_2)$ and $v(f_1)+v(f_2)=v(f_3)$ for
    $f_i\in k((\Group))$.
    The relation $v(f_1)<v(f_2)$ is definable in $L_{\text{ring}}$, since it means $\frac{f_1}{f_2}\in k[[\Group]]$.
     The relation  $v(f_1)+v(f_2)=v(f_3)$ can be written as $v(f_1f_2)=v(f_3)$, that is $\neg( v(f_1f_2)<v(f_3))\wedge\neg( v(f_1f_2)>v(f_3))$.
 \subsection{The impossibility of an $ \exists\forall\exists$ and 
 	$\forall\exists\forall$-definition for $A$}  				
          	\label{The-impossibility-of-an-exists-forall-exists-definition-for-A}  				
          In the previous subsection, we proved that the valuation ring $A$ of the valued field $\mathcal{K}$ is definable in  $L_{\text{ring}}$.  
          In this subsection, we show that even so, it is \textbf{not} definable by any $\exists\forall\exists$ formula. 
          In a very similar way, one can show that $A$ is also not definable by any $\forall\exists\forall$-formula.
          \par 
          Recall that
           by Lemma~\ref{general form of prestel theorem}, the valuation ring $A$ is definable by an $\exists\forall\exists$-formula if and only if the following condition holds:
          \begin{quote}
          If $h: \mathcal{K} \to \mathcal{K}$ is an $\exists\forall$-embedding, then $h(A) \subseteq A$.
          \end{quote}
          We claim that there is an $\exists\forall$-embedding $h:\mathcal{K}\to \mathcal{K}$ such that $h(A)\not\subseteq A$.
          By Theorem~\ref{ax- asli} to find $h$ it suffices to find $\exists\forall$-embeddings 
          $f: \Group\to \Group$ and $g:k\to k$. We let $f$ be the $\exists\forall$-embedding $f_1$
          constructed in Subsection \ref{Recalling-the-construction-in} and $g$ be the identity, and also let $h:\mathcal{K}\to \mathcal{K}$ be the embedding obtained
          by these two maps and application of Theorem~\ref{ax- asli}.
It is clear that
\[
h(A)=k((f(\Lambda_1)))[[f(\Lambda_2)]].
\]
By construction 
$f(\Lambda_1)\not\subseteq \Lambda_1$, and hence $h(A)\not\subseteq A$.
\paragraph{Acknowledgment.}
We would like to thank Arno Fehm for initially suggesting the problem to the master’s students Yadegari and Shirani. As the problem developed further, the two additional authors also joined the project.

   	   	 \end{document}